\RequirePackage{fix-cm}
\documentclass[smallextended]{svjour3}       % onecolumn (second format)
\smartqed  % flush right qed marks, e.g. at end of proof
\usepackage{graphicx}
\usepackage{amsthm}

\newcommand{\bkm}{B_{k,m}}

\newtheorem{thm}{Theorem}
\newtheorem{lem}[thm]{Lemma}
\newtheorem{cor}[thm]{Corollary}

\usepackage{amsmath,amssymb}
\usepackage{amsfonts}

\begin{document}

\title{Mixed, Multi-color, and Bipartite Ramsey Numbers Involving Trees of Small Diameter}
%\subtitle{Do you have a subtitle?\\ If so, write it here}

\titlerunning{Mixed Ramsey Numbers}        % if too long for running head

\author{Jeremy F. Alm       \and
        Nicholas Hommowun   \and
        Aaron Schneider
}

%\authorrunning{Short form of author list} % if too long for running head

\institute{Jeremy F. Alm, Nicholas Hommowun,   and
        Aaron Schneider  \at
              Department of Mathematics, Illinois College, 1101 W. College Ave.,Jacksonville, IL 62650 \\
              Tel.: 217-245-3468\\
 \email{alm.academic@gmail.com}           %  \\
%      %       \emph{Present address:} of F. Author  %  if needed
%           \and
%          Nicholas Hommowun  \at
%
%              Department of Mathematics\\ Illinois College\\ 1101 W. College Ave.\\Jacksonville, IL 62650 \\
%              Tel.: 217-245-3468\\
%           \and
%               Aaron Schneider \at
%
%              Department of Mathematics\\ Illinois College\\ 1101 W. College Ave.\\Jacksonville, IL 62650 \\
%              Tel.: 217-245-3468\\
}

\date{Received: date / Accepted: date}
% The correct dates will be entered by the editor

\maketitle

\begin{abstract}
In this paper we study Ramsey numbers for trees of diameter 3
(bistars) vs., respectively, trees of diameter 2 (stars), complete
graphs, and many complete graphs.  In the case of bistars vs. many
complete graphs, we determine this number exactly as a function of
the Ramsey number for the complete graphs.  We also determine the
order of growth of the bipartite $k$-color Ramsey number for a
bistar. 
% \PACS{PACS code1 \and PACS code2 \and more}
% \subclass{MSC code1 \and MSC code2 \and more}
\end{abstract}

\section{Introduction}

\subsection{Background}

In this paper we investigate Ramsey numbers, both classical and
bipartite, for trees vs.~other graphs.  Trees have been studied less
than other graphs, although there have been a number of papers in
the last few years.  Some general results applying to all trees are
known, such as the following result of Gy\'{a}rf\'{a}s and Tuza
\cite{Gyarfas87}.

\begin{thm}
  Let $T_n$ be a tree with $n$ edges.  Then $R_k(T_n)\leq (n-1)(k+\sqrt{k(k-1)})+2$.
\end{thm}

More recently, various researchers have studied particular trees of
small diameter.  Burr and Roberts \cite{Burr73} completely determine
the Ramsey number $R(S_{n_1},\ldots,S_{n_i})$ for any number of
\emph{stars}, i.e., trees of diameter 2.  Boza et.~al.~\cite{Boza10}
determine $R(S_{n_1},\ldots,S_{n_i},K_{m_1},\ldots,K_{m_j})$ exactly
as a function of $R(K_{m_1},\ldots,K_{m_j})$.  Bahls and Spencer
\cite{Bahls13} study $R(C,C)$, where $C$ is a caterpillar, i.e., a
tree whose non-leaf vertices form a path.  They prove a general
lower bound, and prove exact results in several cases, including
``regular" caterpillars, in which all non-leaf vertices have the
same degree.

We will study bistars (i.e.~trees of diameter 3) vs.~stars and
bistars vs. complete graphs in Section \ref{S2}, bistars vs.~many
complete graphs in Section \ref{S3}, and bistars vs.~bistars in
bipartite graphs in Section \ref{S4}.

\subsection{Notation}

For graphs $G_1,\ldots,G_n$, let $R(G_1,\ldots, G_n)$ denote the
least integer $N$ such that any edge-coloring of $K_N$ in $n$ colors
must contain, for some $1\leq i\leq n$,   a monochromatic $G_i$ in
the $i^{\text{th}}$ color.  Let $S_n$ denote the $(n+1)$-vertex
graph consisting of a vertex $v$ of degree $n$ and $n$ vertices of
degree 1 (a star).  Let $B_{k,m}$ denote the $(k+m)$-vertex graph
with a vertex $v$ of degree $k$, a vertex $w$ incident to $v$ of
degree $m$, and $k+m-2$ vertices of degree 1 (a \emph{bistar}).  We
will call the edge $vw$ the \emph{spine} of $B_{k,m}$.  (Note that
some authors refer to the set of vertices $\{ v,w\}$ as the spine.)
We will depict the spine of a bistar with a double-struck edge; see
Figure \ref{fig:F0}.

\begin{figure}[htb!]
\centering
\includegraphics[width=120pt]{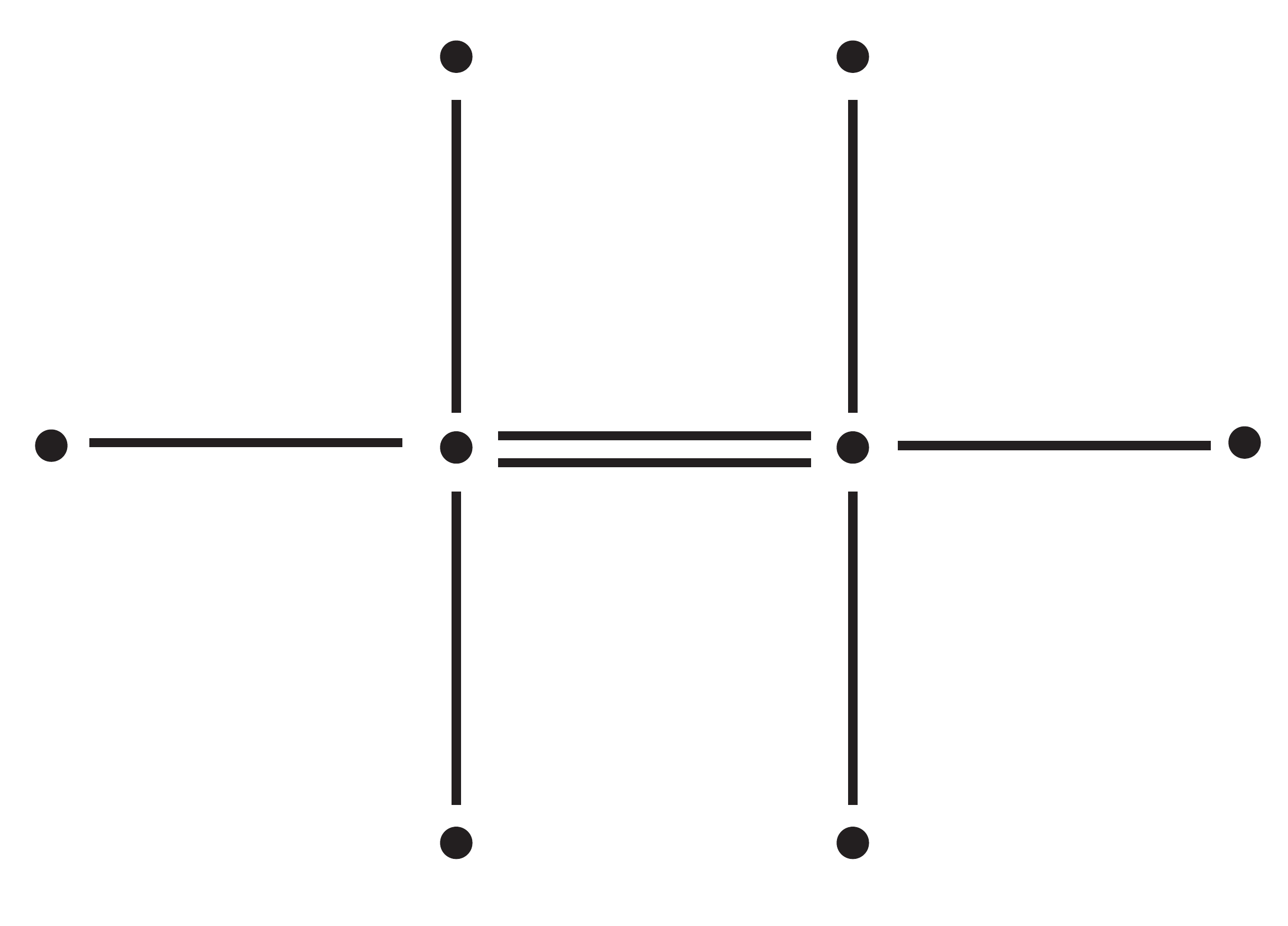}
\caption{A bistar, with spine indicated} \label{fig:F0}
\end{figure}

For a graph $G$ whose edges are colored red and blue, and for
vertices $v$ and $w$, if $v$ and $w$ are incident by a red edge, we
will say (for the sake of brevity) that $w$ is a ``red neighbor" of
$v$.  Let $\text{deg}_{\text{red}}(v)$ denote the number of red
neighbors of $v$, and let $$\Delta_{\text{red}}(G)=\text{max}\{
\text{deg}_\text{red}(v):v\in G\}$$ and
$$\delta_{\text{red}}(G)=\text{min}\{ \text{deg}_\text{red}(v):v\in
G\}.$$

In Section \ref{S2} we will make use of \emph{cyclic colorings}.
Let $K_N$ have vertex set $\{ 0,1,2,\ldots,N-1\}$, and let
$R\subseteq \mathbb{Z}_N\backslash 0$ such that $R=-R$, i.e., $R$ is
closed under additive inverse.  Define a coloring of $K_N$ by
\[
  uv\textrm{ is colored \emph{red} if }u-v\in R\textrm{ and \emph{blue} otherwise.}
\]

Cyclic colorings are computationally nice.  For instance, it is not
hard to show that if $R\subseteq R+R$, then any two vertices $v$ and
$w$ incident by a red edge must share a red neighbor.  We will need
this fact in the proof of Theorem \ref{LB}.

\section{Mixed 2-Color Ramsey Numbers}\label{S2}

First we consider bistars vs. stars.  We have the following easy
upper bound.
\begin{thm}\label{UB}
  $R(\bkm,S_n)\leq k+m+n-1$.
\end{thm}

\begin{proof}
  Let $N=k+m+n-1$, and let the edges of $K_N$ be colored in red and blue.  Suppose this coloring contains no blue $S_n$. Then every red edge is the spine of a red $\bkm$, as follows.

  If there is no blue $S_n$, then $\Delta_{\text{blue}}\leq n-1$, and hence $\delta_{\text{red}}\geq(N-1)-(n-1)=k+m-1$.  Let the edge $uv$ be colored red.  Then both $u$ and $v$ have $(k-1)+(m-1)$ red neighbors besides each other.  Even if these sets of neighbors coincide, we may select $k-1$ leaves for $u$ and $m-1$ leaves for $v$, giving a red $B_{k,m}$.
\end{proof}

The following lower bound uses some cyclic colorings.

\begin{thm}\label{LB}
  $R(\bkm,S_n)>\lfloor\frac{k+m}{2}\rfloor+n$ for $k,m\geq 4$.
\end{thm}

\begin{proof}
  Let $k+m$ be odd.  Let $N=\lfloor\frac{k+m}{2}\rfloor+n$.  Let $G$ be any $(n-1)$-regular graph on $N$ vertices.  Consider the edges of $G$ to be the \emph{blue} edges, and replace all non-edges of $G$ with red edges, so that the resulting $K_N$ is $\lfloor\frac{k+m}{2}\rfloor$-regular for red.  Clearly, this coloring admits no blue $S_n$.  Consider the red edge set.  If an edge $uv$ is colored red, then $u$ and $v$ combined have at most $k+m-3$ red neighbors besides each other, which is not enough to supply the needed $k-1$ red leaves for $u$ and the $m-1$ red leaves for $v$.

  Now let $k+m$ be even, and $N=\frac{k+m}{2}+n$.  We seek a subset $R\subseteq\mathbb{Z}_N$ that is symmetric $(R=-R)$ and of size $\frac{k+m}{2}$ satisfying $R\subseteq R+R$.  Thus each vertex will have red degree $\frac{k+m}{2}$, but any red edge $uv$ cannot be the spine of a red $\bkm$, since $u$ and $v$ will have a common neighbor.  There are two cases:

  Case (i.): $\frac{k+m}{2}$ is even.  Let $R'=\{ 2\}\cup\{ 2\ell+1:1\leq\ell\leq\frac{k+m-4}{4}\}$, and let $R:=R'\cup -R'$.  It is easy to check that $R\subseteq R+R$.  Setting $B=\mathbb{Z}_N\backslash\{ R\cup 0\}$, we have $|B|=n-1$, and so the cyclic coloring of $K_N$ induced by $R$ and $B$ has no red $\bkm$ and no blue $S_n$.

  Case (ii.): $\frac{k+m}{2}$ is odd.  Let $R'=\{ 2\}\cup\{ 2\ell+1:1\leq\ell\leq\frac{k+m-6}{4}\}$, and set $R:=R'\cup\{\frac{k+m}{2}\}\cup -R'$.  Again, set $B=\mathbb{Z}_N\backslash\{ R\cup 0\}$, and the cyclic coloring of $K_N$ induced by $R$ and $B$ has the desired properties.
\end{proof}

\begin{cor} \label{cor}
  $R(B_{n,n},S_n)>2n$ for $n\geq 4$.
\end{cor}
We conjecture that the lower bound in   Corollary \ref{cor} is
tight; that is, that $R(B_{n,n},S_n)=2n+1$ for $n\geq 4$.  We show
that this result obtains for $n=4$ (but not for $n=3$).

\begin{thm}
  $R(B_{3,3},S_3)=6$.
\end{thm}

\begin{proof}
  A lower bound is supplied by the classic critical coloring of $K_5$ for $R(3,3)$.  See Figure \ref{fig:FP}.

  \begin{figure}[htb!]
\centering
\includegraphics[width=90pt]{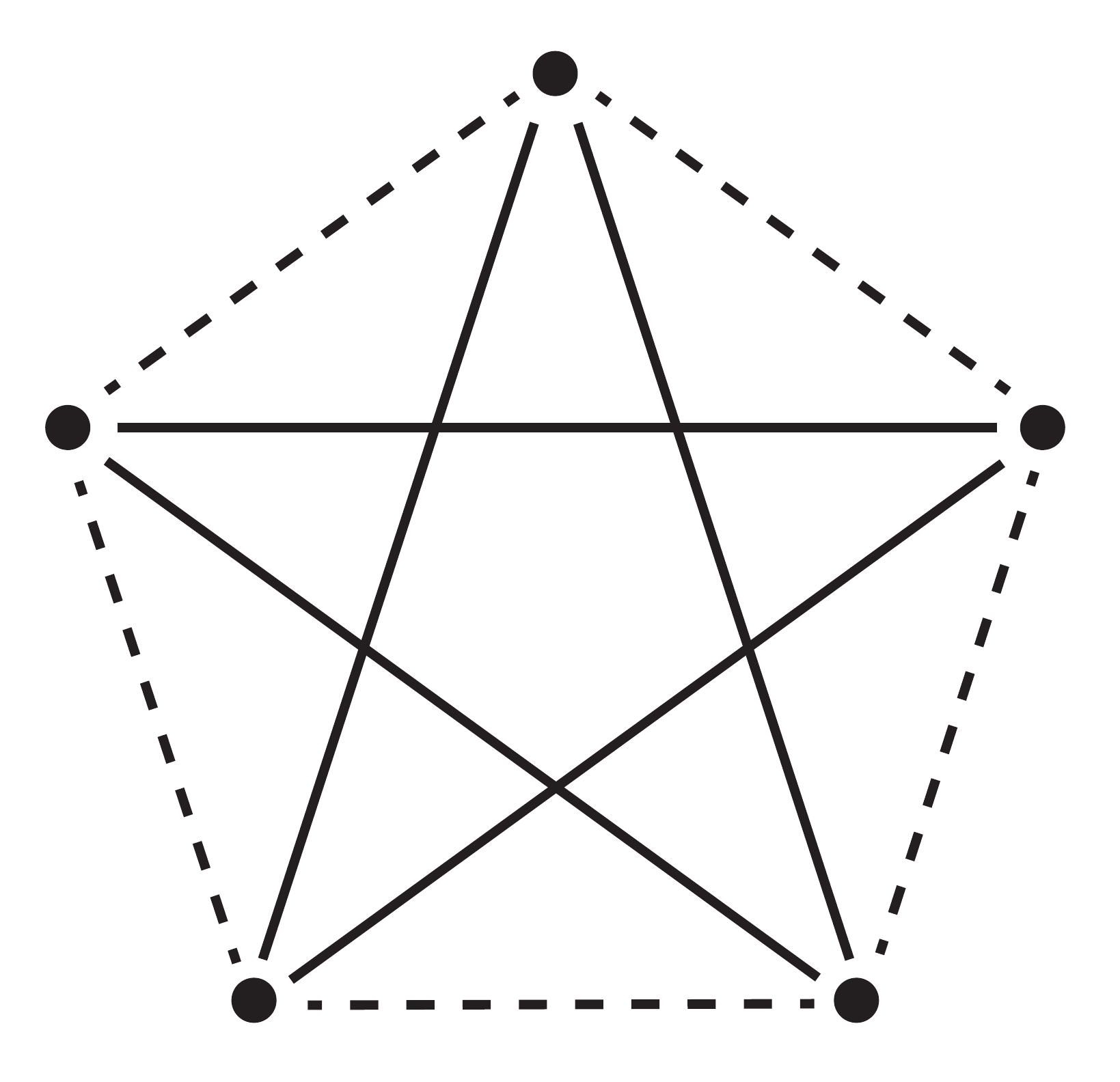}
\caption{Critical coloring of $K_5$} \label{fig:FP}
\end{figure}

  For the upper bound, suppose there exists a 2-coloring of $K_6$ with no blue $S_3$.  Then $\delta_\text{red}\geq 3$.  So consider the red subgraph $G$.  If $G$ has a vertex of degree 5, the existence of a $B_{3,3}$ is immediate.  If $G$ has a vertex of degree 4, then it must have 2 such vertices $u$ and $v$.  If $u\nsim v$, then $G$ must look like Figure \ref{fig:F1}.
\begin{figure}[htb!]
\centering
\includegraphics[width=120pt]{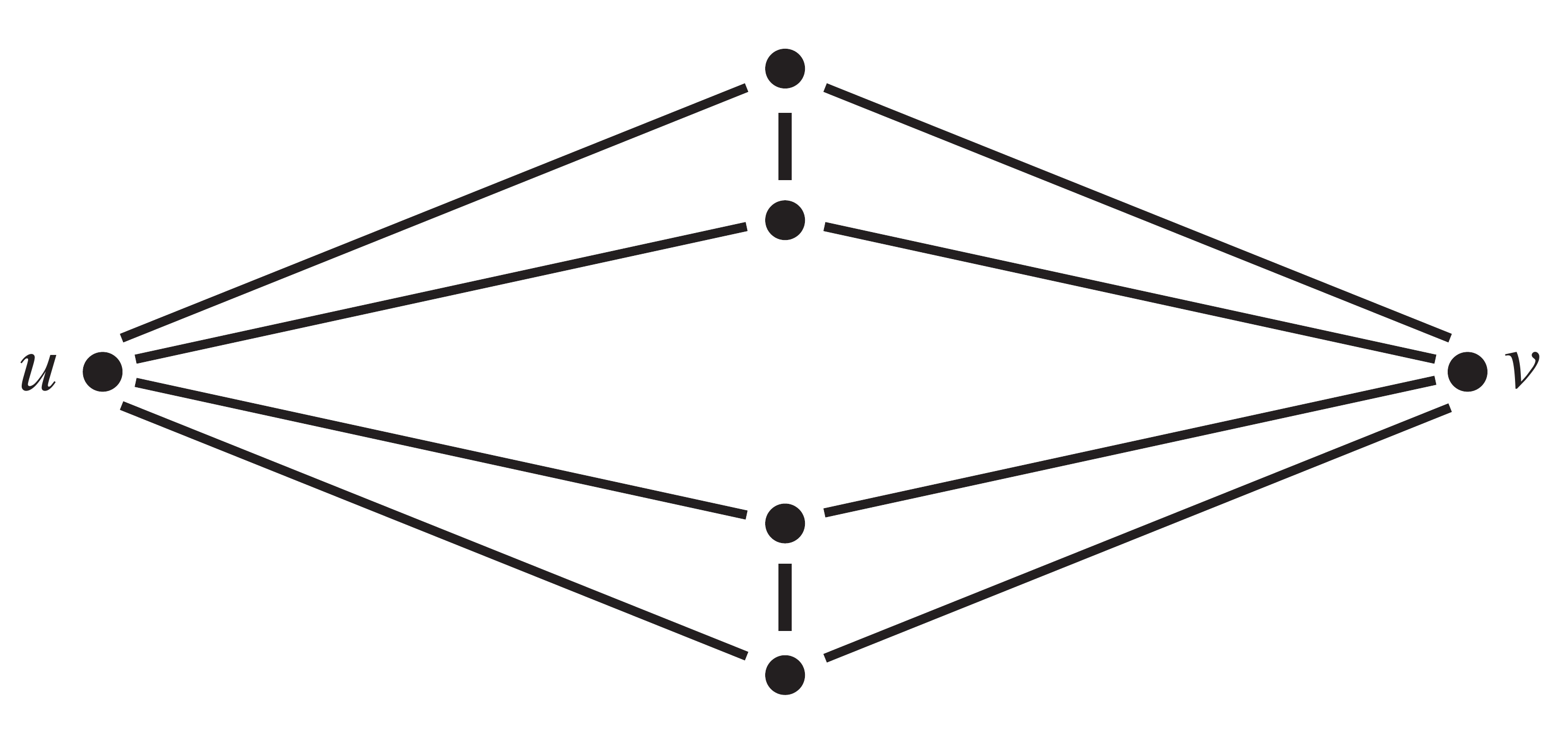}
\caption{Configuration of the red subgraph $G$.} \label{fig:F1}
\end{figure}
  One may use any edge incident to $u$ or $v$ as a spine.

  If $u\sim v$, and $u$ and $v$ do \emph{not} share all three remaining neighbors, then the existence of a $B_{3,3}$ is immediate.  So suppose $u$ and $v$ have neighbors $x,y$, and $z$.  The only way for $G$ to have degree sequence $(4,4,3,3,3,3)$ is for the remaining vertex $w$ to be adjacent to $x,y$, and $z$.  Then we have a $B_{3,3}$ as indicated in Figure \ref{fig:F2}.
\begin{figure}[htb!]
\centering
\includegraphics[width=60pt]{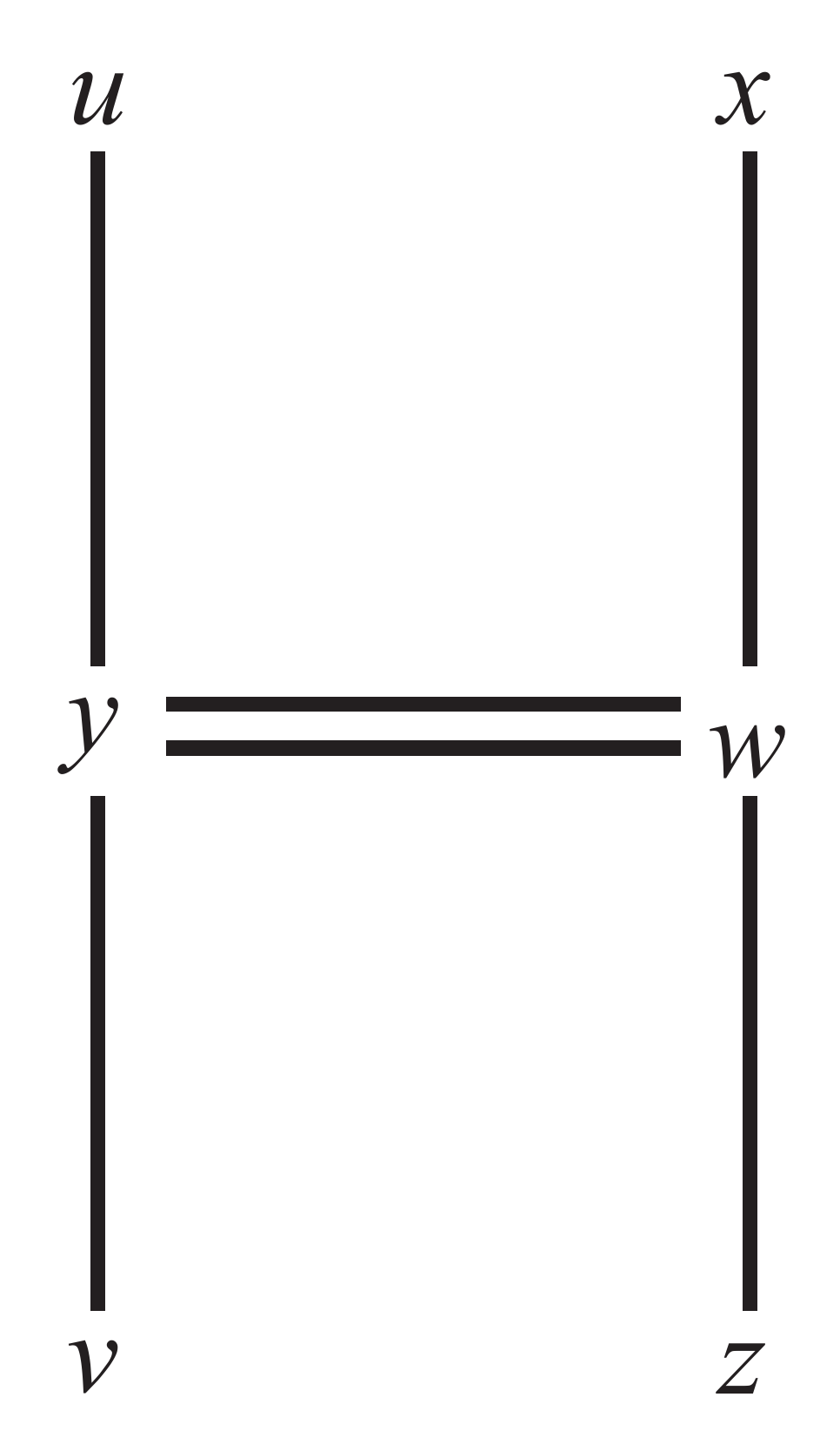}
\caption{A red $B_{3,3}$} \label{fig:F2}
\end{figure}
  Finally, suppose $G$ is 3-regular.  If there is no $B_{3,3}$, then any adjacent vertices share a neighbor.  It is not hard to see that adjacent vertices cannot share \emph{two} neighbors in a 3-regular graph on 6 vertices. If any two adjacent vertices share exactly one neighbor, then   $G$ can be partitioned into edge-disjoint triangles.  But any vertex in such a graph must have even degree, since its degree will be twice the number of triangles in which it participates.  This is a
  contradiction, so there must be some vertices $u$ and $v$ that
  have no common neighbor. But $u$ and $v$ each have degree three,
  immediately yielding a $B_{3,3}$.
\end{proof}

\begin{thm}
  $R(B_{4,4},S_4)=9$.
\end{thm}

\begin{proof}
  The lower bound is given by Theorem \ref{LB}.  For the upper bound, suppose a 2-coloring of $K_9$ contains no blue $S_4$. Then $\delta_\text{red}\geq 5$.  Let $G$ be the red subgraph.  Since $G$ has odd order, there must be at least one vertex $v$ of degree $\geq 6$.  Suppose $v\sim w$.  It is easy to see that $v$ and $w$ must have at least two neighbors in common; call them $y$ and $z$.  Now $v$ is adjacent to 3 other vertices; call them $x_1,x_2$, and $x_3$.  There are two remaining vertices $x_4$ and $x_5$.  If $w$ is adjacent to either of them, we are done.  So suppose $w$ is adjacent to $x_1$ and $x_2$.  If either $x_4$ or $x_5$ is adjacent to $v$, we are done, so suppose neither $x_4$ nor $x_5$ is adjacent to $v$ or to $w$.  Then $x_4$ (in order to have degree $\geq$ 5) must be adjacent to $y_1$ or to $y_2$.  Suppose it's $y_1$.  There are two cases:
  \begin{enumerate}
    \item $y_1\sim x_5$.  Then we have a red $B_{4,4}$ as indicated in Figure \ref{fig:F3}.

    \begin{figure}[htb!]
\centering
\includegraphics[width=120pt]{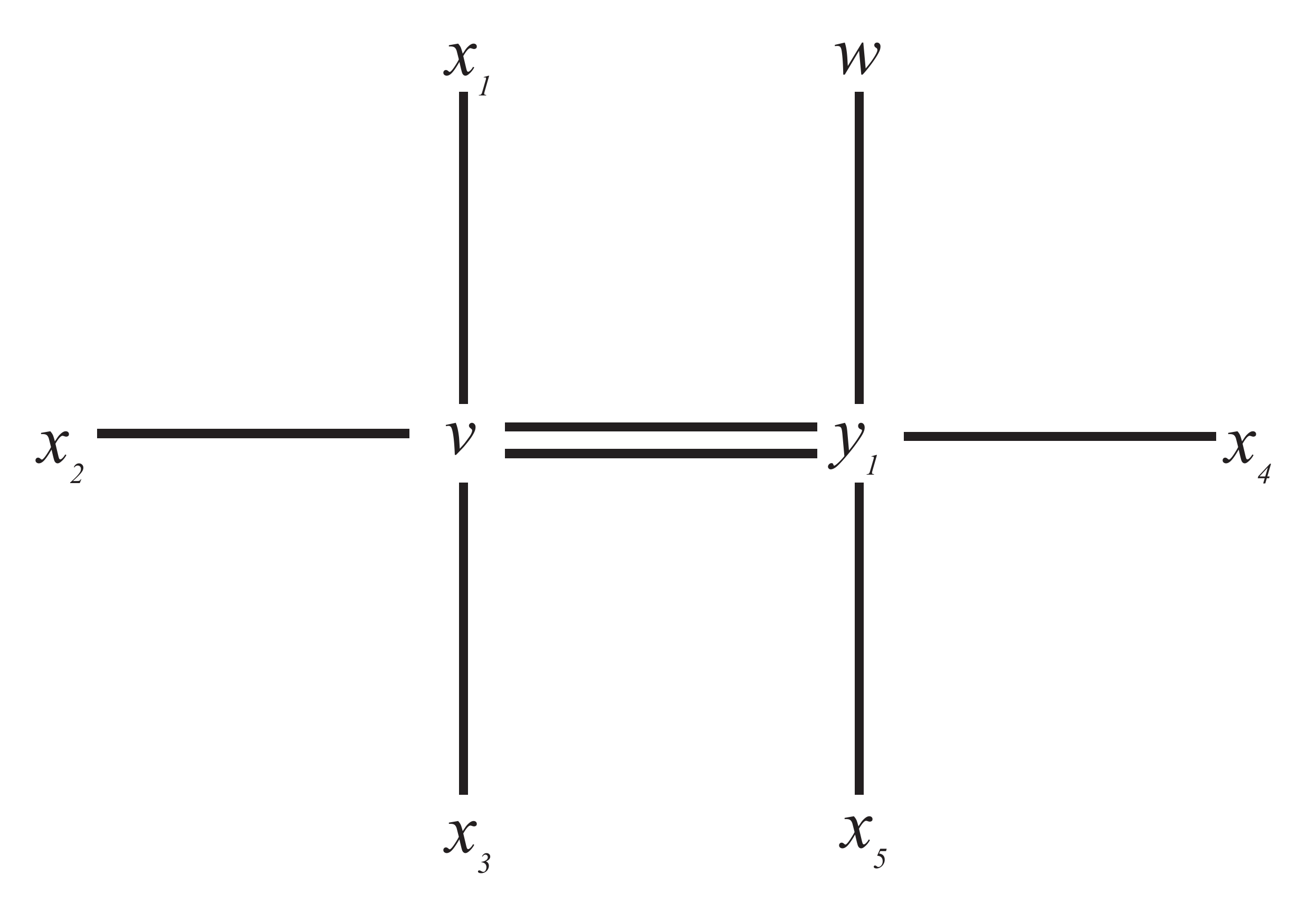}
\caption{A red $B_{4,4}$} \label{fig:F3}
\end{figure}

    \item $y_1\nsim x_5$.  Then, since deg$(y_1)\geq 5$, $y_1\sim x_i$ for some $i\in\{ 1,2,3\}$.  Then we have a red $B_{4,4}$ as indicated in Figure \ref{fig:F4},  where $|\{ i,k,\ell\}|=3$.

    \begin{figure}[htb!]
\centering
\includegraphics[width=120pt]{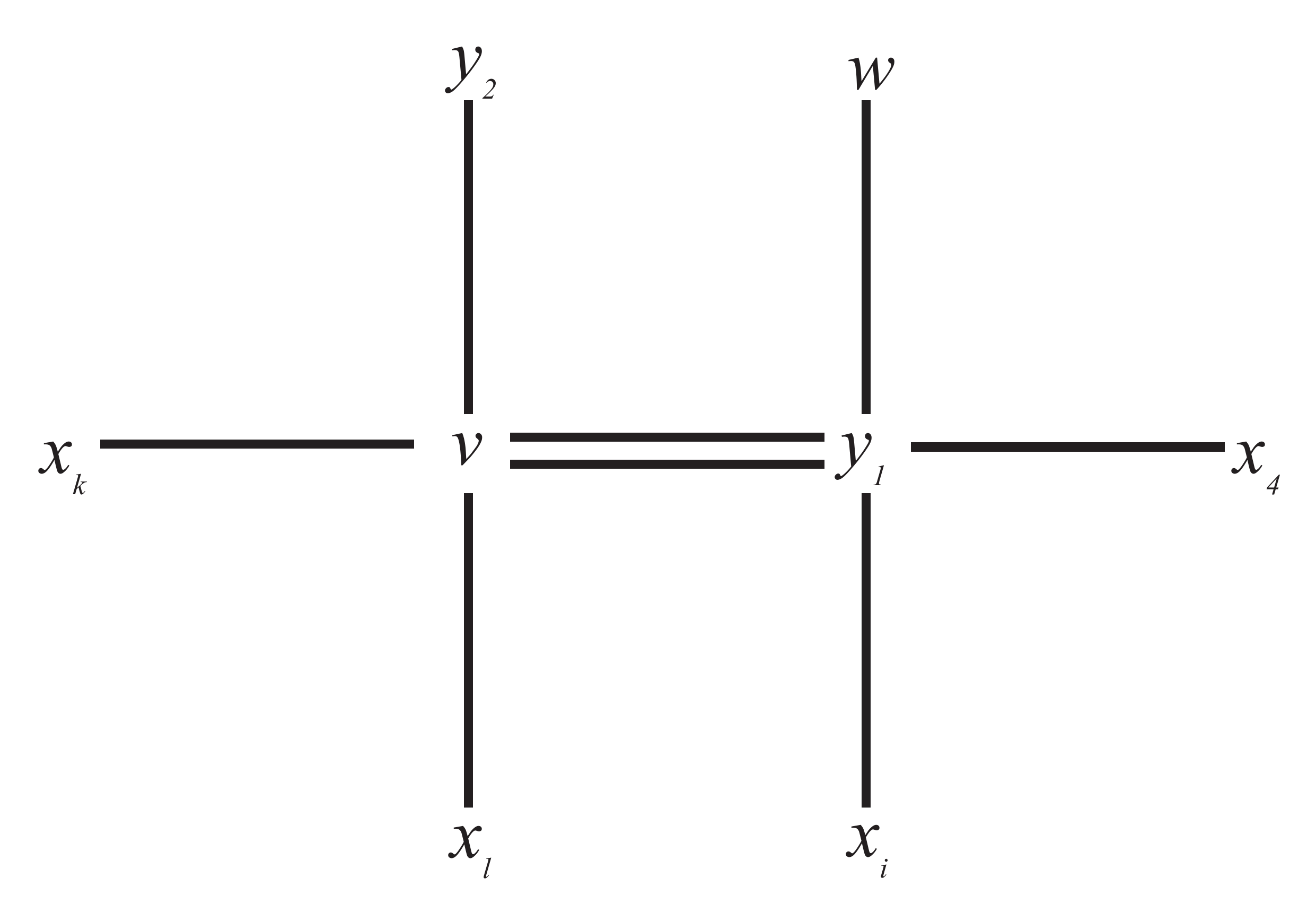}
\caption{A red $B_{4,4}$} \label{fig:F4}
\end{figure}

  \end{enumerate}
  \end{proof}

  Consideration of $R(B_{5,5},S_5)$ leads into rather unpleasant case analysis when trying to reduce the upper bound from that given by Theorem \ref{UB}.

  Now we consider bistars vs. complete graphs.

  \begin{thm}\label{base}
    $R(\bkm,K_3)=2(k+m-1)+1$.
  \end{thm}

  \begin{proof}
    For the lower bound, let $V_1$ and $V_2$ be two red cliques, each of size $k+m-1$, and let every edge between $V_1$ and $V_2$ be colored blue.

    For the upper bound, let $N=2(k+m-1)+1$, and give $K_N$ an edge-coloring in red and blue.  Suppose there is a vertex $v$ with blue degree at least $k+m$.  If any edge in $N_\text{blue}(v)$ is blue, we have a blue triangle.  If not, then $N_\text{blue}(v)$ is a red clique of size at least $k+m$, so it contains a red $\bkm$.

    So then suppose that $\Delta_\text{blue}<k+m$.  It follows that $\delta_\text{red}\geq k+m-1$.  Then every red edge is the spine of some red $\bkm$.  To see this, let $uv$ be colored red.  Both $u$ and $v$ each have at least $k+m-2$ other red neighbors.  Even if these red neighborhoods coincide, there are still $k-1$ red leaves for $u$ and $m-1$ red leaves for $v$.
  \end{proof}

  Now we extend to arbitrary $K_n$.

  \begin{thm} \label{Kn}
    $R(\bkm,K_n)=(k+m-1)(n-1)+1$.
  \end{thm}

  \begin{proof}
    We proceed by induction on $n$.  Theorem \ref{base} provides the base case $n=3$.

    So assume $n>3$, and let $R(\bkm,K_{n-1})\leq(k+m-1)(n-2)+1$.  Let $N=(k+m-1)(n-1)+1$, and consider any edge-coloring of $K_n$ in red and blue.  If $\delta_\text{red}\geq k+m-1$, then every red edge is the spine of a red $\bkm$, so suppose $\delta_\text{red}\leq k+m-2$.  Then there is a vertex $v$ with blue degree at least $(k+m-1)(n-2)+1$.  By the induction hypothesis, the subgraph induced by $N_\text{blue}(v)$ contains either a red $\bkm$ or a blue $K_{n-1}$.  In the latter case, the blue $K_{n-1}$ along with $v$ forms a blue $K_n$.

    For the lower bound, let $V_1,\ldots,V_{n-1}$ be vertex-disjoint red cliques, each of size $k+m-1$.  Color all edges among the $V_i$'s blue.  Clearly there are no red $\bkm$'s.  Since the blue subgraph forms a Turan graph, there are no blue $K_n$'s.
  \end{proof}

\section{Mixed Multi-color Ramsey Numbers}\label{S3}

In \cite{Boza10}, the authors determine
$R(S_{k_1},\ldots,S_{k_i},k_{n_i},\ldots,K_{n_\ell})$ exactly as a
function of $R(K_{n_1},\ldots,K_{n_\ell})$.    In \cite{Omidi11},
Omidi and Raeisi give a shorter proof of this result via the
following lemma, whose proof is straight from The Book.

\begin{lem}\label{lem}
  Let $G_1,\ldots,G_m$ be connected graphs, let $r=R(G_1,\ldots,G_m)$ and $r'=R(K_{n_1},\ldots,K_{n_\ell})$.  If $n\geq 2$ and $R(G_1,\ldots,G_m,K_n)=(r-1)(n-1)+1$, then $R(G_1,\ldots,G_m,K_{n_1},\ldots,K_{n_\ell})=(r-1)(r'-1)+1$.
\end{lem}

\begin{proof}
  Let $R=R(G_1,\ldots,G_m,K_{n_1},\ldots,K_{n_\ell})$.  For the lower bound, give $K_{r'-1}$ an edge-coloring in $\ell$ colors $\beta_1,\ldots,\beta_\ell$ that has no copy of $K_{n_i}$ in color $\beta_i$.  Replace each vertex of $K_{r'-1}$ by a complete graph of order $r-1$ whose edges are colored by colors $\alpha_1,\ldots,\alpha_m$ so that no copy of $G_i$ appears in color $\alpha_i$.  Each edge in the original graph $K_{r'-1}$ expands to a copy of $K_{r-1,r-1}$, with each edge the same color as the original edge.  This shows that $R>(r-1)(r'-1)$.

  For the upper bound, let $N=(r-1)(r'-1)+1$, and color the edges of $K_N$ in colors $\alpha_1,\ldots,\alpha_m,\beta_1,\ldots,\beta_\ell$.  Recolor the edges colored $\beta_1,\ldots,\beta_\ell$ with a new color $\alpha$.  Since $R(G_1,\ldots,G_m,K_{r'})= (r'-1)(r-1)+1=N$, $K_N$ contains a copy of $G_i$ in color $\alpha_i$ or a copy of $K_{r'}$ in color $\alpha$.  In the former case we are done, so assume the latter obtains.  Then consider the clique $K_{r'}$ which is colored $\alpha$.  Return to the original coloring in colors $\beta_1,\ldots,\beta_\ell$. Since $R(K_{n_1},\ldots,K_{n_\ell})=r'$, some color class $\beta_i$ contains a copy of $K_{n_i}$.  This concludes the proof.
\end{proof}

We will now make use of  Lemma \ref{lem} to determine
$R(\bkm,K_{n_1},\ldots,K_{n_\ell})$ as a function of
$R(K_{n_1},\ldots,K_{n_\ell})$.

\begin{thm}
  $R(\bkm,K_{n_1},\ldots,K_{n_\ell})=(k+m-1)[R(K_{n_1},\ldots,K_{n_\ell})-1]+1$.
\end{thm}

\begin{proof}
  From Theorem \ref{Kn} we have that $R(\bkm, K_n)=(k+m-1)(n-1)+1$.  Note that $R(\bkm, K_2)=k+m$, so that

  \begin{align*}
    R(\bkm,K_2,K_n) &=R(\bkm,K_n)\\
    &=[R(\bkm,K_2)-1](n-1)+1.
  \end{align*}
  Hence we may apply Lemma \ref{lem} to get $R(\bkm,K_{n_1},\ldots,K_{n_\ell})=(k+m-1)[R(K_{n_1},\ldots,K_{n_\ell})-1]+1$.
\end{proof}

The authors are unsure whether a similar result can be proved for
multiple bistars; we leave this as an open problem.

\section{Bipartite Ramsey Numbers}\label{S4}

Let $G_1$ and $G_2$ be bipartite graphs.  Then $BR(G_1,G_2)$ is the
least integer $N$ so that any 2-coloring of the edges of $K_{N,N}$
contains either a red $G_1$ or a blue $G_2$.  In \cite{Hattingh13},
Hattingh and Joubert determine the bipartite Ramsey number for
certain bistars:

\begin{thm}\label{bi}
  Let $k,n\geq 2$.  Then $BR(B_{k,k},B_{n,n})=k+n-1$.
\end{thm}

We generalize this result slightly.

\begin{thm}
  Let $k\geq m\geq 2$, $n\geq\ell\geq 2$.  Then $BR(\bkm,B_{n,\ell})=k+n-1$.
\end{thm}

\begin{proof}
  The upper bound follows immediately from Theorem \ref{bi}.  The lower bound construction given in Theorem 1 of Hattingh-Joubert for $BR(B_{s,s},B_{t,t})$ does not work for us.  We need this construction:  Let $L$ and $R$ be the partite sets, and let $N=k+n-2=(k-1)+(n-1)$.  Let $L=\{ v_0,v_1,\ldots,v_{N-1}\}$ and $R=\{w_0,w_1,\ldots,w_{N-1}\}$.  Color $v_iw_j$ \emph{red} if $(i-j)\mod N\in\{ 0,1,\ldots,k-2\}$, and \emph{blue} if $(i-j)\mod N\in\{k-1,\ldots,N-1\}$.  Then the red subgraph is $(k-1)$-regular, hence no red $\bkm$, and the blue subgraph is $(n-1)$-regular, hence no blue $B_{n,\ell}$.
\end{proof}

\begin{cor}
  Let $T_m$ (resp., $T_n$) be a tree of diameter at most 3 with maximum degree $m$ (resp., $n$).  Then $BR(T_m,T_n)=m+n-1$.
\end{cor}

Hattingh and Joubert also prove the following $k$-color upper bound.

\begin{thm}
For $k\geq 2$ and $m\geq 3$, we have

   $$BR_k(B_{m,m})=BR(B_{m,m},\ldots,B_{m,m})\leq \left\lceil k(m-1) +\sqrt{(m-1)^2(k^2-k)-k(2m-4)}\right\rceil$$
\end{thm}

Hence $BR_k(B_{m,n})=O(k)$.  We provide a lower bound to get the
following result.

\begin{thm}
  Fix $m\geq 3$.  Then $BR_k(B_{m,m})=\Theta(k)$.
\end{thm}

\begin{proof}
  We show that $BR_k(B_{m,m})>k\cdot (m-1)$.  Let $N=k\cdot(m-1)$, and consider a $k$-coloring of the edges of $K_{N,N}$ in colors $c_0,\ldots,c_{k-1}$.  Let the partite sets be $L=\{ v_0,\ldots,,v_{N-1}\}$ and $R=\{ w_0,\ldots, w_{N-1}\}$.  Color edge $v_iw_j$ with color $c_\ell$ if and only if $\ell\equiv(i-j)\mod k$.  Then the $c_\ell$-subgraph is $(m-1)$-regular, hence there can be no monochromatic $B_{m,m}$.
\end{proof}

% For one-column wide figures use
%\begin{figure}
%% Use the relevant command to insert your figure file.
%% For example, with the graphicx package use
%  \includegraphics{example.eps}
%% figure caption is below the figure
%\caption{Please write your figure caption here}
%\label{fig:1}       % Give a unique label
%\end{figure}
%%
%% For two-column wide figures use
%\begin{figure*}
%% Use the relevant command to insert your figure file.
%% For example, with the graphicx package use
%  \includegraphics[width=0.75\textwidth]{example.eps}
%% figure caption is below the figure
%\caption{Please write your figure caption here}
%\label{fig:2}       % Give a unique label
%\end{figure*}
%%
%% For tables use
%\begin{table}
%% table caption is above the table
%\caption{Please write your table caption here}
%\label{tab:1}       % Give a unique label
%% For LaTeX tables use
%\begin{tabular}{lll}
%\hline\noalign{\smallskip}
%first & second & third  \\
%\noalign{\smallskip}\hline\noalign{\smallskip}
%number & number & number \\
%number & number & number \\
%\noalign{\smallskip}\hline
%\end{tabular}
%\end{table}

%\begin{acknowledgements}
%If you'd like to thank anyone, place your comments here
%and remove the percent signs.
%\end{acknowledgements}

% BibTeX users please use one of
%\bibliographystyle{spbasic}      % basic style, author-year citations
%\bibliographystyle{spmpsci}      % mathematics and physical sciences
%%\bibliographystyle{spphys}       % APS-like style for physics
%\bibliography{../../masterrefs}   % name your BibTeX data base

\end{document}